\documentclass[12pt]{article}
\usepackage{amssymb, amsmath, amsthm, amsfonts,xcolor, enumerate}
\usepackage{tikz}
\usetikzlibrary{patterns}
\usetikzlibrary{shapes}
\usepackage{fullpage}
\usepackage{hyperref}
\newtheorem{theorem}{Theorem}[section]
\newtheorem{lemma}[theorem]{Lemma}
\theoremstyle{definition}

\newtheorem{conj}[theorem]{Conjecture}

\theoremstyle{definition}

\theoremstyle{definition}

\theoremstyle{definition}

\theoremstyle{definition}

\theoremstyle{definition}

\theoremstyle{definition}

\theoremstyle{definition}

\newcommand\ex{\ensuremath{\mathrm{ex}}}

\title{On the number of $K_4$-saturating edges}
\author{
J\'ozsef Balogh
\thanks{Department of Mathematical Sciences,
University of Illinois at Urbana-Champaign, Urbana, Illinois 61801, USA {\tt jobal@math.uiuc.edu}. Research is partially supported by Simons Fellowship, NSF CAREER Grant
DMS-0745185 and Arnold O.
Beckman Research Award (UIUC Campus Research Board 13039).} 
\quad\quad
Hong Liu
\thanks{Department of Mathematical Sciences,
University of Illinois at Urbana-Champaign, Urbana, Illinois 61801, USA {\tt hliu36@illinois.edu}.}
}

\begin{document}
\maketitle

\begin{abstract}
Let $G$ be a $K_4$-free graph, an edge in its complement is a $K_4$-\emph{saturating} edge if the addition of this edge to $G$ creates a copy of $K_4$. Erd\H{o}s and Tuza conjectured that for any $n$-vertex $K_4$-free graph $G$ with $\lfloor n^2/4\rfloor+1$ edges, one can find at least $(1+o(1))\frac{n^2}{16}$ $K_4$-saturating edges. We construct a graph with only $\frac{2n^2}{33}$ $K_4$-saturating edges. Furthermore, we prove that it is best possible, i.e., one can always find at least $(1+o(1))\frac{2n^2}{33}$ $K_4$-saturating edges in an $n$-vertex $K_4$-free graph with $\lfloor n^2/4\rfloor+1$ edges.
\end{abstract}
\section{Introduction}
The notation in this paper is standard. For a graph $G$, denote by $\overline{G}$ its complement. For vertex subsets $U,W\subseteq V(G)$, denote $N(U):=\bigcap_{v\in U} N(v)$ and $E(U,W)$ the set of cross edges between $U$ and $W$.

Mantel~\cite{Mantel} showed that the maximum number of edges in an $n$-vertex triangle-free graph is $\lfloor n^2/4\rfloor$. Rademacher (unpublished) extended this result by showing that any $n$-vertex graph with $\lfloor n^2/4\rfloor+t$ edges contains at least $t\lfloor n/2\rfloor$ triangles, for $t=1$. Lov\'asz and Simonovits~\cite{L-S}, improving Erd\H{o}s~\cite{E-R},  proved this for every $t\le n/2$. Erd\H{o}s~\cite{E-clique} showed analogue results for cliques and Mubayi~\cite{Mu,Mu2} for color-critical graphs and for some hypergraphs.

In general, we call Erd\H{o}s-Rademacher-type problem the following: for any extremal question, what is the number of forbidden configurations appearing in a graph somewhat denser than the extremal graph? This type of problems have been studied in various contexts: A \emph{book} of size $q$ consists of $q$ triangles sharing a common edge. Khad\v{z}iivanov and Nikiforov~\cite{K-N}, answering a question of Erd\H{o}s, showed that any $n$-vertex graph with $\lfloor n^2/4\rfloor+1$ edges contains a book of size at least $n/6$. In the context of Sperner's Theorem, Kleitman~\cite{Kl}, answering a question of Erd\H{o}s and Katona, determined the minimum number of $2$-chains a poset must contain if its size is larger than its largest anti-chain. Recently, this theorem was extended to $k$-chains by Das, Gan and Sudakov~\cite{D-G-S}.

Let $G$ be an $n$-vertex $K_4$-free graph, an edge in $\overline{G}$ is a $K_4$-\emph{saturating} edge if the addition of this edge to $G$ creates a copy of $K_4$. Denote by $f(G)$ the number of $K_4$-saturating edges in $\overline{G}$ and by $f(n,e)$ the maximum integer $\ell$ such that every $n$-vertex $K_4$-free graph with $e$ edges must have at least $\ell$ $K_4$-saturating edges. The first extremal result related to clique-saturating edges was by Bollob\'as~\cite{BB} who proved that if every edge in $\overline{G}$ is a $K_r$-saturating edge, then $e(G)\ge{n\choose 2}-{n-r+2\choose 2}$ and this bound is best possible. Later it was extended by Alon~\cite{A}, Frankl~\cite{F} and Kalai~\cite{K} using linear algebraic method. Recently, saturation problems were phrased in the language of `graph bootstrap percolation', see~\cite{BBM} and~\cite{BBMR}.

In the case of $K_4$, Bollob\'as' example is the following: let $F$ be an $n$-vertex $K_4$-free graph with two vertices adjacent to all other vertices which form an independent set. This graph has only linear many edges, $e(F)=2n-3$, and yet all edges in $\overline{F}$ are $K_4$-saturating edges. To the other extreme, $K_{\lceil n/2\rceil,\lfloor n/2\rfloor}$ shows that a graph could have up to $\lfloor n^2/4\rfloor$ edges with no $K_4$-saturating edge, i.e. $f(n,\lfloor n^2/4\rfloor)=0$. Erd\H{o}s and Tuza~\cite{E-conj} conjectured that if a $K_4$-free graph $G$ has $\lfloor n^2/4\rfloor+1$ edges, then suddenly there are quadratic many $K_4$-saturating edges. This conjecture can be considered, as formulated before, an Erd\H{o}s-Rademacher-type problem concerning the number of $K_4$-saturating edges.
\begin{conj}[Erd\H{o}s-Tuza~\cite{E-conj}]
$$f(n,\lfloor n^2/4\rfloor+1)=(1+o(1))\frac{n^2}{16}.$$
\end{conj}

We disprove this conjecture. We give a counterexample with only $\frac{2n^2}{33}$ $K_4$-saturating edges. Furthermore, we prove that $(1+o(1))\frac{2n^2}{33}$ is best possible, that is, one can always find at least $(1+o(1))\frac{2n^2}{33}$ $K_4$-saturating edges in an $n$-vertex $K_4$-free graph with $\lfloor n^2/4\rfloor+1$ edges.

\begin{theorem}\label{2/33}
For $n\ge 73$,
$$f(n, \lfloor n^2/4\rfloor+1)=\frac{2n^2}{33}+O(n).$$
\end{theorem}

We shall prove the following theorem, which implies the lower bound in Theorem~\ref{2/33}.
\begin{theorem}\label{hoho}
Let $G$ be an $n$-vertex $K_4$-free graph with $\lfloor n^2/4\rfloor$ edges, for $n\ge 73$. If $G$ contains a triangle, then
$$f(G)\ge \frac{2n^2}{33}-\frac{3n}{11}.$$
This is best possible when $n$ is divisible by 66.
\end{theorem}

\begin{proof}
[Proof of Theorem~\ref{2/33}] The upper bound is by the construction described in Section~\ref{upper}. For the lower bound, let $G$ be a $K_4$-free graph with $\lfloor n^2/4\rfloor+1$ edges. By Mantel's theorem, it contains a triangle. Let $G'$ be a subgraph obtained from $G$ by removing an edge such that $G'$ contains a triangle. By Theorem~\ref{hoho}, $f(G')\ge \frac{2n^2}{33}-\frac{3n}{11}$. The relation $f(G)\ge f(G')$ completes the proof.
\end{proof}

\noindent\textbf{Remark:} (i) Slight modification of our proof gives the following stability result: Given any $K_4$-free graph $G$ with $(1-o(1))n^2/4$ edges, if $G$ contains a triangle, then $f(G)\ge (1-o(1))2n^2/33$.

\medskip

\noindent (ii) Unlike the case about the number of triangles in~\cite{E-R} and~\cite{L-S}, where every additional edge, up to $n/2$, gurantees $\lfloor n/2\rfloor$ additional triangles, in our problem, even with linear many extra edges, the number of $K_4$-saturating edges is still at most $(1+o(1))2n^2/33$. In particular, $f\left(n,\lfloor\frac{n^2}{4}\rfloor+t\right)=\frac{2n^2}{33}+O(n)$ for $1\le t\le \frac{n}{66}$.

\medskip

\noindent (iii) One might define a $K_{r}$-saturating edge of a graph $G$, for $r\ge 5$, as we did for $K_4$. Denote by $\ex(n,K_{r-1})$ the maximum size of an $n$-vertex $K_{r-1}$-free graph. We think that a similar phenomenon holds: if $G$ is $K_{r}$-free and $e(G)=\ex(n,K_{r-1})+1$, then the number of $K_{r}$-saturating edges is at least $\left(\frac{2(r-3)^2}{(r-1)(4r^2-19r+23)}+o(1)\right)n^2$. A straightforward generalization of our construction shows that if the conjecture is true, then it is best possible. Some of the ideas of our proof works for $r\ge 5$ as well, but some does not.

\medskip

The paper is organized as follows: We give a construction for the upper bound in Theorem~\ref{2/33} and an extremal example for Theorem~\ref{hoho} in Section~\ref{upper}. The proof for Theorem~\ref{hoho} is given in Section~\ref{lower}.  We will omit floors and ceilings when it is not critical and we make no effort optimizing some of the constants.
\section{Upper bound constructions}\label{upper}
Fix an integer $n$ divisible by 66. We present an $n$-vertex $K_4$-free graph $H$ with $n^2/4+n/66$ edges and $f(H)=2n^2/33-7n/33$. Note that from this graph one can easily remove $n/66-1$ edges without changing the number of $K_4$-saturating edges.
We also give an extremal example showing the bound in Theorem~\ref{hoho} is best possible.

\noindent\textbf{Construction for Theorem~\ref{2/33}:} To construct $H$, start with a $C_5$ on $\{v_1,v_2,v_3,v_4,v_5\}$ with a chord $v_1v_3$. Blow up each $v_i$ to an independent set $V_i$ of the following size: $|V_1|=|V_3|=16n/66$, $|V_2|=4n/66+1$, $|V_4|=15n/66$ and $|V_5|=15n/66-1$, see Figure~\ref{fig-0}. Then $H$ is $K_4$-free with $n^2/4+n/66$ edges. The only $K_4$-saturating edges are those in $V_1,V_2,V_3$, which gives $f(H)=2n^2/33-7n/33$.

\begin{figure}[ht]
\begin{center}
\includegraphics{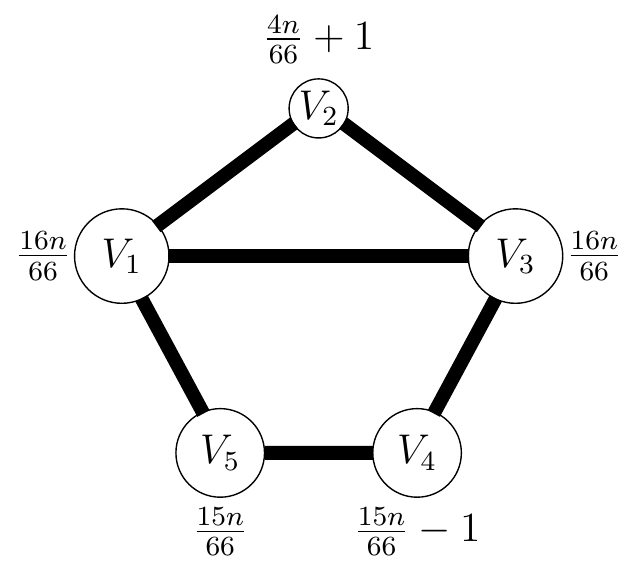}
\end{center}
\caption{A $K_4$-free graph $H$ with $e(H)=\frac{n^2}{4}+\frac{n}{66}$ and $f(H)=\frac{2n^2}{33}-\frac{7n}{33}$.}
\label{fig-0}
\end{figure}

\noindent\textbf{Construction for Theorem~\ref{hoho}:} Define $H'$ the same way as $H$, except that $|V_2'|=4n/66$ and $|V_4'|=15n/66$. This graph is $K_4$-free with $n^2/4$ edges and $f(H')= \frac{2n^2}{33}-\frac{3n}{11}$.

\section{Proof of Theorem~\ref{hoho}}\label{lower}
Let $G$ be a $K_4$-free graph with $n^2/4$ edges and containing a triangle. Fix, in $G$, a maximum family of vertex-disjoint triangles, say $\mathcal{T}=\{T_1,T_2,...,T_{tn}\}$, where $0< t\le 1/3$. We write $V(\mathcal{T})$ for $\bigcup_{i=1}^{tn}V(T_i)$, $E(\mathcal{T})$ for $E(G[V(\mathcal{T})])$ and $e(\mathcal{T}):=|E(\mathcal{T})|$. Let $G'=G-V(\mathcal{T})$, since $\mathcal{T}$ is of maximum size, $G'$ is a $K_3$-free graph with $e(G')\le\frac{(1-3t)^2n^2}{4}$. Denote by  $r_1n^2$ the number of $K_4$-saturating edges incident to $V(\mathcal{T})$, and by $r_2n^2$ the number of $K_4$-saturating edges in $V(G')$. Hence $f(G)=(r_1+r_2)n^2$. First we give a lower bound on $r_1$.

\begin{lemma}\label{r1}
$$r_1n^2\ge \left(\frac{1}{4}-t+\frac{3t^2}{2}\right)n^2-e(G')-\frac{3}{2}tn\ge \left(\frac{t}{2}-\frac{3t^2}{4}\right)n^2-\frac{3}{2}tn.$$
\end{lemma}
\begin{proof}
Let $t_i=e(T_i,G\setminus \bigcup_{j=1}^{i}T_j)$, clearly $\sum_{i=1}^{tn} t_i=e(G)-e(G')-3tn$. Since $G$ is $K_4$-free, every vertex can have at most two neighbors in each triangle. Thus $t_i-(n-3i)$ is a lower bound on the number of vertices in $G\setminus \bigcup_{j=1}^{i}T_j$ having degree 2 in $T_i$, each of which gives a $K_4$-saturating edge. Indeed, say $V(T_1)=\{x,y,z\}$, and $w\in N(x)\cap N(y)$, then $wz$ is a $K_4$-saturating edge. Thus,
\begin{eqnarray*}
r_1n^2&\ge& \sum_{i=1}^{tn} (t_i-(n-3i))=(e(G)-e(G')-3tn)-\left(tn^2-3\frac{tn(tn+1)}{2}\right)\\
&\ge&\left(\frac{1}{4}-t+\frac{3t^2}{2}\right)n^2-e(G')-\frac{3}{2}tn\ge \left(\frac{t}{2}-\frac{3t^2}{4}\right)n^2-\frac{3}{2}tn,
\end{eqnarray*}
where the last inequality follows from $e(G')\le\frac{(1-3t)^2n^2}{4}$.
\end{proof}

Let $T_i\in\mathcal{T}$ be a triangle in $\mathcal{T}$. Denote by $N_j(T_i)\subseteq V(G')$, for $0\le j\le 3$, the set of vertices in $G'$ that has exactly $j$ neighbors in $T_i$. Since $G$ is $K_4$-free, $N_3(T_i)=\emptyset$, for every $T_i$'s. Further define $p_0(T_i)=\frac{|N_0(T_i)|}{n}$, $p_1(T_i)=\frac{|N_1(T_i)|}{n}$ and $p_2(T_i)=\frac{|N_2(T_i)|}{n}$. Thus by definition, $p_0(T_i)+p_1(T_i)+p_2(T_i)=1-3t$.

The next lemma shows that there is a triangle $T\in\mathcal{T}$ with large $|N_2(T)|$.

\begin{lemma}\label{bigtriangle}
There exists a triangle $T\in \mathcal{T}$, such that\\
(i) $e(T,G')\ge \left(\frac{3}{2}-\frac{21t}{4}\right)n$, and\\
(ii) $p_2(T)\ge \frac{1}{2}-\frac{9t}{4}+p_0(T)$.
\end{lemma}
\begin{proof}
(i) The edge set of $G$ can be partitioned into $E(G'), E(\mathcal{T}, G')$ and $E(\mathcal{T})$. Notice that since $G$ is $K_4$-free, there are at most $6$ edges between any pair of triangles in $\mathcal{T}$. Hence $e(\mathcal{T})\le 3tn+6{tn\choose 2}=3t^2n^2$.

Thus we have $e(\mathcal{T},G')=e(G)-e(G')-e(\mathcal{T})\ge \frac{n^2}{4}-\frac{(1-3t)^2n^2}{4}-3t^2n^2\ge \left(\frac{3t}{2}-\frac{21t^2}{4}\right)n^2$. Therefore, there exists a triangle $T\in\mathcal{T}$ with $e(T,G')\ge e(\mathcal{T},G')/(tn)\ge \left(\frac{3}{2}-\frac{21t}{4}\right)n$.

\noindent (ii) Let $T\in\mathcal{T}$ be a triangle satisfying (i). Note that $2p_2(T)+p_1(T)=\frac{e(T,G')}{n}$. Using $p_0(T)+p_1(T)+p_2(T)=1-3t$, we have $p_2(T)-p_0(T)\ge \frac{3}{2}-\frac{21t}{4}-(1-3t)=\frac{1}{2}-\frac{9t}{4}$.
\end{proof}

From now on, we let $T=\{x,y,z\}$ be a triangle in $\mathcal{T}$ sending the most edges to $G'$, hence
it has the two properties of Lemma~\ref{bigtriangle}. For brevity we write $p_j=p_j(T)$ and $N_i=N_i(T)$ for $0\le j\le 2$. Furthermore, define $A=N_{G'}(xy)$, $B=N_{G'}(yz)$ and $C=N_{G'}(xz)$. Note that $A,B,C$ are pairwise disjoint independent sets, otherwise $T\cup A\cup B\cup C$ contains a copy of $K_4$. Define $N_x:=N_{G'}(x),N_y:=N_{G'}(y)$ and $N_z:=N_{G'}(z)$. Let $a=\frac{|A|}{|N_2|}$, $b=\frac{|B|}{|N_2|}$ and $c=\frac{|C|}{|N_2|}$, thus $a+b+c=1$. For $1\le k\le 3$, we say that $T$ spans a $k$-\emph{joint-book}, if among $A,B,C$, exactly $3-k$ of them are empty sets.

\begin{lemma}\label{3-book}
If $T$ spans a 3-joint-book, then we have
$$r_2n^2\ge \frac{1}{6}\left[\frac{3}{2}-\frac{21t}{4}\right]^2n^2-e(\overline{G'})-(1-3t)n.$$
\end{lemma}
\begin{proof}
First notice that $N_x,N_y$ and $N_z$ are all independent sets. Indeed, suppose $N_x$ contains an edge, then $T\cup N_x\cup B$ contains two vertex-disjoint triangles, contradicting the maximality of $\mathcal{T}$.

Note that ${|N_x|\choose 2}+{|N_y|\choose 2}+{|N_z|\choose 2}\le r_2n^2+e(\overline{G'})$. Indeed, every pair of vertices in $N_x,N_y$ or $N_z$ gives a non-edge in $G'$ and those $K_4$-saturating edges in $A,B,C$ are counted twice. Additionally, $|N_x|+|N_y|+|N_z|=e(T,G')\ge\left(\frac{3}{2}-\frac{21t}{4}\right)n$, and $e(T,G')\le 2(1-3t)n$. Thus,
\begin{eqnarray*}
r_2n^2+e(\overline{G'})&\ge& {|N_x|\choose 2}+{|N_y|\choose 2}+{|N_z|\choose 2} \ge  3{e(T,G')/3\choose 2}\\
&=&\frac{1}{6}(e(T,G'))^2-\frac{1}{2}e(T,G')\ge\frac{n^2}{6}\left[\frac{3}{2}-\frac{21t}{4}\right]^2-(1-3t)n.
\end{eqnarray*}
\end{proof}

We first show that if $T$ spans a 3-joint-book, then $f(G)\ge 2n^2/33-3n/11$.

\begin{lemma}\label{no-3-book}
For $n\ge 73$, if $T$ spans a 3-joint-book, then $f(G)\ge \frac{2n^2}{33}-\frac{3n}{11}$.
\end{lemma}
\begin{proof}
Note that $e(G')+e(\overline{G'})=\frac{(1-3t)^2n^2}{2}-\frac{(1-3t)n}{2}$. By Lemmas~\ref{r1} and~\ref{3-book}, we have
\begin{eqnarray*}
f(G)&=& (r_1+r_2)n^2\ge \left(\frac{1}{4}-t+\frac{3t^2}{2}\right)n^2-e(G')-\frac{3}{2}tn\\
&+&\frac{1}{6}\left[\frac{3}{2}-\frac{21t}{4}\right]^2n^2-e(\overline{G'})-(1-3t)n\\
&\ge& \left(\frac{51t^2}{32}-\frac{5t}{8}+\frac{1}{8}\right)n^2-\frac{n}{2}\ge\frac{13n^2}{204}-\frac{n}{2}\ge\frac{2n^2}{33}-\frac{3n}{11},
\end{eqnarray*}
since $\frac{51t^2}{32}-\frac{5t}{8}+\frac{1}{8}\ge \frac{13}{204}$ when $0<t\le 1/3$, and the last inequality holds for $n\ge 73$.
\end{proof}

\begin{proof}[Proof of Theorem~\ref{hoho}]
By Lemma~\ref{no-3-book}, we may assume that $T$ spans a $k$-joint-book with $k\le 2$. Without loss of generality assume that $B=\emptyset$, i.e. $b=0$. Then $a+c=1$ and $|A|+|C|=p_2n$. Notice that each pair of vertices in $A$ and $C$ is a $K_4$-saturating edge, hence
\begin{eqnarray}\label{r2}
r_2n^2\ge {|A|\choose 2}+{|C|\choose 2}\ge 2{p_2n/2\choose 2}=\frac{p_2^2}{4}n^2-\frac{p_2n}{2}.
\end{eqnarray}
If $t\ge \frac{1}{5}$, then Lemma~\ref{r1} implies $f(G)\ge r_1n^2\ge \left(\frac{t}{2}-\frac{3t^2}{4}\right)n^2-\frac{n}{2}\ge \frac{2n^2}{33}$ for $n\ge 54$. Thus we may assume that $t<\frac{1}{5}$. The right hand side in~\eqref{r2} is minimized when $p_2$ is at its lower bound provided by Lemma~\ref{bigtriangle}, as $\frac{1}{2}-\frac{9t}{4}>\frac{1}{n}$ for $n\ge 20$. Hence
$$r_2n^2\ge \frac{1}{4}\left(\frac{1}{2}-\frac{9t}{4}\right)^2n^2-\frac{1}{2}\left(\frac{1}{2}-\frac{9t}{4}\right)n.$$
Therefore using Lemma~\ref{r1}, we have
\begin{eqnarray*}
f(G)&=& (r_1+r_2)n^2\ge\left(\left(\frac{t}{2}-\frac{3t^2}{4}\right)+\frac{1}{4}\left(\frac{1}{2}-\frac{9t}{4}\right)^2\right)n^2-\frac{1}{2}\left(3t+\frac{1}{2}-\frac{9t}{4}\right)n\\
&=&\left(\frac{33t^2}{64}-\frac{t}{16}+\frac{1}{16}\right)n^2-\frac{1}{2}\left(\frac{3t}{4}+\frac{1}{2}\right)n\ge \frac{2n^2}{33}-\frac{3n}{11}-\frac{3}{44},
\end{eqnarray*}
where the function on the right hand side is minimized at $t=\frac{2}{33}+\frac{4}{11n}$. Since both $tn$ and $f(G)$ are integers, checking all $n$ modulo 33, we have
$$f(G)\ge \frac{2n^2}{33}-\frac{3n}{11}.$$
\end{proof}

\end{document}